\documentclass[11pt]{article}
\usepackage[english]{babel}
\usepackage[latin1]{inputenc}
\usepackage{latexsym,amsfonts,amsmath,amssymb,amsthm,euscript,epsfig,dsfont,graphicx,color,float}
\usepackage[active]{srcltx}
\usepackage[T1]{fontenc}
\usepackage[round,comma]{natbib}
\usepackage{mycommands}
\usepackage{url}

\renewcommand{\textsl}{\emph}

\title{A Primal Condition \\ for Approachability with Partial Monitoring}
\author{Shie Mannor --- Vianney Perchet --- Gilles Stoltz
\thanks{This work began after an interesting objection of and further
discussions with Jean-François Mertens
during the presentation of our earlier results~\citep{Ext} on the dual characterization of approachability with partial monitoring
at the conference \textsl{Games Toulouse 2011}. The material presented in this article was developed further with Sylvain Sorin back in Paris, whom we thank deeply for his advice and encouragements.
Sadly, Jean-François Mertens, a close collaborator of Sylvain Sorin, passed away in the mean time.
This contribution is in honor of both of these important contributors to the theory of repeated games.} \vspace{.3cm} \\
\hspace{-1cm}
Technion --- Université Paris-Diderot --- CNRS / Ecole normale supérieure / HEC Paris}

\DeclareMathOperator{\cav}{cav}

\newcommand\cH{\mathcal{H}}
\newcommand\bH{\mathbf{H}}

\newcommand\bh{\mathbf{h}}

\newcommand\cC{\mathcal{C}}
\newcommand\cD{\mathcal{D}}

\newcommand\cI{\mathcal{I}}
\newcommand\cJ{\mathcal{J}}
\newcommand\cK{\mathcal{K}}
\newcommand\cL{\mathcal{L}}
\newcommand\cF{\mathcal{F}}
\newcommand\cS{\mathcal{S}}

\newcommand\N{\mathbb{N}}
\newcommand\R{\mathbb{R}}
\renewcommand\L{\mathbb{L}}

\newcommand\om{\overline{m}}
\newcommand\tr{\widetilde{r}}
\newcommand\ts{\widetilde{s}}

\renewcommand{\d}{\mathrm{d}}
\renewcommand{\P}{\mathbb{P}}

\newcommand{\Full}{\mathrm{Full}}
\newcommand{\Dark}{\mathrm{Dark}}
\renewcommand{\leq}{\leqslant}

\renewcommand{\geq}{\geqslant}

\renewcommand{\mathbf}{\boldsymbol}
\renewcommand{\epsilon}{\varepsilon}
\newcommand{\Chs}{\cC_{\mathrm{\tiny hs}}}
\newcommand{\Corth}{\cC_{\mathrm{\tiny orth}}}
\newcommand{\Cpoly}{\cC_{\mathrm{\tiny polyt}}}
\newcommand{\Chsind}[1]{\cC_{\mathrm{\tiny hs},\,#1}}
\renewcommand{\textsl}{\emph}
\renewcommand{\bar}{\overline}
\newcommand{\NR}{\mathrm{NR}}

\newtheorem{definition}{Definition}
\newtheorem{theorem}{Theorem}
\newtheorem{proposition}{Proposition}
\newtheorem{lemma}{Lemma}
\newtheorem{rema}{Remark}
\newtheorem{exam}{Example}
\newtheorem{corollary}{Corollary}
\newenvironment{remark}[1]{\begin{rema} \em}{\end{rema}}
\newenvironment{example}[1]{\begin{exam} \em}{\end{exam}}

\renewcommand{\citealt}{\citealp}

\begin{document}
\maketitle
\thispagestyle{empty}

\begin{abstract}
In approachability with full monitoring there are two types of conditions that are known to be equivalent for convex sets: a primal and a dual condition. The primal one is of the form: a set $\cC$ is approachable if and only all containing half-spaces are
approachable in the one-shot game. The dual condition is of the form: a convex set $\cC$ is approachable if
and only if it intersects all payoff sets of a certain form. We consider approachability in games with partial monitoring.
In previous works~\citep{Per11,Ext} we provided a dual characterization of approachable convex sets
and we also exhibited efficient strategies in the case where $\cC$ is a polytope.
In this paper we provide primal conditions on a convex set to be approachable with partial monitoring.
They depend on a modified reward function and lead to approachability strategies
based on modified payoff functions and that proceed by projections similarly to Blackwell's (1956) strategy.
This is in contrast with previously studied strategies in this context that relied mostly
on the signaling structure and aimed at estimating well the distributions of the signals received.
Our results generalize classical results by~\cite{Koh75} (see also~\citealt{MeSoZa94})
and apply to games with arbitrary signaling structure as well as to arbitrary convex sets.
\end{abstract}

\section{Introduction}

Approachability theory dates back to the seminal paper of \cite{Bla56}. In this paper Blackwell presented conditions under which a player can guarantee that the long-term average vector-valued reward is asymptotically close to some target set regardless of the opponent actions.
If this property holds, we say that the set is approachable.
In the full monitoring case studied in \cite{Bla56} there are two equivalent conditions for a convex set to be approachable.
The first one, known as a primal condition (or later termed the ``B'' condition in honor for
Blackwell) states that every half-space that contains the target set is also approachable. It turns out that whether a half-space is approachable is determined by the sign of the value of some associated zero-sum game.
The second characterization, known as the dual condition, states that for every mixed action of the opponent, the player can guarantee that the
one-shot vector-valued reward is inside the target set.

Approachability theory has found many applications in game theory, online learning, and related fields. Both primal and dual characterizations are of interest therein. Indeed, checking if the dual condition holds is formally simple while a concrete approaching strategy naturally derives from the primal condition (it only requires solving a one-shot zero-sum game at every stage of the repeated vector-valued game).
\medskip

Approachability theory has been applied to zero-sum repeated games with incomplete information and/or imperfect (or partial) monitoring. The work of \cite{Koh75} (see also \citealt{MeSoZa94}) uses approachability to derive strategies for games with incomplete information. The general case of repeated vector-valued games with partial monitoring was studied only recently. A dual characterization of approachable convex sets with partial monitoring was presented by~\cite{Per11}. However, it is not useful for deriving  approaching strategies since it essentially requires to run a calibration algorithm, which is known to be computationally hard.
In a recent work (\citealt{Ext}) we derived efficient strategies for approachability in games with partial monitoring
in some cases, e.g., when the convex set to be approached is a polytope.
However, these strategies are based on the dual condition, and not on any primal one: they thus do not shed light on the structure of the game.
\medskip

In this paper we provide a primal condition for approachability in games with partial monitoring. It can be stated, as in \cite{Bla56}, as a requirement that every half-space containing the target set is one-shot approachable. However, the reward function has to be modified in some
cases for the condition to be sufficient. We also show how it leads to an efficient approachability strategy, at least in the case of approachable polytopes.

\paragraph{Outline.}
In Section~\ref{sec:model} we define a model of partial monitoring and recall some of the basic results from approachability (both in terms of its primal and dual characterizations). In Section~\ref{sec:relatedandobjective} we explain the current state-of-the-art, recall the dual condition for approachability with partial monitoring, and outline our objectives.
In Section~\ref{sec:halfspaces} we provide results for approaching half-spaces as they have the simplest characterization of approachability:
we show that the signaling structure has no impact on approachability of half-spaces, only the payoff structure does. This is not the case anymore for approachability of more complicated convex sets, which is the focus of the subsequent sections.
In Section~\ref{sec:prodmeasureorthants} we discuss the case of a target set that is an orthant under additional properties on the payoff--signaling structure: we show that a natural primal condition holds.
This condition, which we term the ``upper-right-corner property'' is the main technical contribution of the paper. We show that basically, we can study approachability for a modified payoff function and that under some favorable conditions, a primal condition is easy to derive (which is the main conceptual contribution).
As an intermezzo,
we link our results to~\cite{Koh75} in Section~\ref{SE:IncompleteInfo} and show that repeated games with imperfect information can be analyzed using our approach for games with imperfect monitoring. In Section~\ref{sec:noproduct} we then analyze the case of a general signaling structure
for the approachability of orthants and provide an efficient approaching strategy based on the exhibited primal conditions.
Finally, we relax the shape of the target set from an orthant to a polytope in Section~\ref{SE:Polyhedra},
and then to a general convex set in Section~\ref{sec:gen}. Our generalizations show that the {\em same} primal condition holds in all cases.
The generalization from orthants to polytopes is based on the observation that any polytope can be represented as an orthant in a space whose dimensionality is the number of linear inequalities describing the polytope and on a modified reward function. The generalization to general convex sets uses support functions and lifting to derive similar results; we provide some background material on support functions in the appendix.

\section{Model and preliminaries}
\label{sec:model}

We now define the model of interest and then recall some basic results from approachability theory for repeated vector-valued games (with full monitoring).

\subsection*{Model and notation}
We consider a vector-valued game between two players, a decision maker (or player) and Nature, with respective finite action sets $\cI$ and $\cJ$, whose
cardinalities are referred to as $I$ and $J$.
We denote by $d$ the dimension of the reward vector
and equip $\R^d$ with the $\ell^2$--norm $\|\,\cdot\,\|_2$.
The payoff function of the player is given by
a mapping $r : \cI \times \cJ \to \R^d$,
which is multi-linearly extended to $\Delta(\cI) \times \Delta(\cJ)$,
the set of product-distributions over $\cI \times \cJ$.

At each round, the player and Nature simultaneously choose their actions $i_n \in \cI$ and $j_n \in \cJ$,
at random according to probability distributions denoted by $x_n \in \Delta(\cI)$ and $y_n \in \Delta(\cJ)$. At the end of a round, the
player does not observe Nature's action $j_n$ nor even the payoff $r_n:=r(i_n,j_n)$ he obtains: he only gets to see some signal.
More precisely, there is a finite set $\cH$ of possible signals, and the signal $s_n \in \cH$ that is shown to the
player is drawn at random according to the distribution
$H(i_n,j_n)$, where the mapping $H : \cI \times \cJ \to \Delta(\cH)$ is known by both players.

The player is said to have full monitoring if $\cH=\cJ$ and $H(i,j) = \Full(i,j):=\delta_j$, i.e., if the action of Nature is observed.
We speak of a game in the dark when the signaling structure $H$ is not informative at all, i.e., when $\cH$ is reduced to a single
signal referred to as $\emptyset$; we denote this situation by $H = \Dark$.
\bigskip

Of major interest will be \emph{maximal information mapping} $\bH: \Delta(\cJ) \to \Delta(\cH)^{\cI}$, which is defined as follows.
The image of each $j \in \cJ$ is the vector $\bH(j) = \bigl( H(i,j) \bigr)_{i \in \cI}$, and this definition is
extended linearly onto $\Delta(\cJ)$. An element of the image $\cF = \bH\bigl( \Delta(\cJ) \bigr)$ of $\bH$ is referred to as a \emph{flag}.
The notion of ``flag'' is key: the player only accesses the mixed actions $y$ of Nature through $H$. Indeed, as is intuitive and as is made
more formal at the end of the proof of Proposition~\ref{PR:Kohlberg}, he could at best
access or estimate the flag $\bH(y)$ but not $y$ itself.

For every $x \in \Delta(\cI)$ and $\bh \in \cF$ the set of payoffs compatible with $\bh$ is
\begin{equation}
\label{def:om}
\om(x,\bh)=\bigl\{ r(x,y): \ y\in \Delta(\cJ)\ \mbox{such that}\ \bH(y)=\bh\bigr\}.
\end{equation}
The set $\om(x,\bh)$ represents all the rewards that are statistically compatible with the flag~$\bh$
(or put differently, the set of all possible rewards we cannot distinguish from).

Note that with full monitoring, $\cH$ reduces to $\Delta(\cJ)$ and one has $\om(x,y) = \bigl\{ r(x,y) \bigr\}$ for all $y \in \Delta(\cJ)$.
\bigskip

Finally, we denote by $M$ a uniform $\ell_2$--bound on $r$, that is,
\[
M = \max_{i,j} \bigl\| r(i,j) \|\,.
\]
Also, for every $n \in \N$ and sequence $(a_m)_{m \in \N}$, the
average of the first $n$ elements is referred to as $\bar{a}_n=(1/n) \sum_{m=1}^n a_m$. The distance to a set $\cC$ is denoted by $d_{\cC}$.

A behavioral strategy $\sigma$ of the player is a mapping from the set of his finite histories $\cup_{n \in \N} \left(\cI\times \cH\right)^n$ into $\Delta(\cI)$. Similarly, a strategy $\tau$ of nature is a mapping from $\cup_{n \in \N} \left(\cI\times \cH \times \cJ\right)^n$ into $\Delta(\cJ)$. As usual, we denote by $\P_{\sigma,\tau}$ the probability induced by the pair $(\sigma,\tau)$ onto $\left(\cI\times\cH\times\cJ\right)^{\mathbb{N}}$.

\subsection*{Definition and some properties of approachability}

A set $\cC \subseteq \R^d$ is $r$--approachable for the signaling structure $H$, or, in short,
is $(r,H)$--approachable, if, for all $\varepsilon >0$, there exists a strategy $\sigma_\epsilon$
of the player and a natural number $N\in\N$ such that, for all strategies $\tau$ of Nature,
\[
\P_{\sigma_{\epsilon},\tau} \bigl\{ \exists \, n \geq N \ \ \mbox{s.t.} \ \ d_{\cC}(\bar{r}_n) \geq \varepsilon \bigr\} \leq \varepsilon\,.
\]
We refer to the strategy $\sigma_\epsilon$ as an $\varepsilon$--approachability strategy of $\cC$.
It is easy to show that the approachability of $\cC$ implies the existence of a strategy ensuring
that the sequence of the average vector-valued payoffs converges to the set $\cC$ almost surely,
uniformly with respect to the strategies of Nature. By analogy such a strategy is called a $0$--approachability strategy of $\cC$.

Conversely, a set $\cC$ is $r$--excludable for the signaling structure $H$ if, for some $\delta>0$,
the complement of the $\delta$--neighborhood of $\cC$ is $r$--approachable by Nature for the signaling structure $H$.
\medskip

\paragraph{Primal characterization.}
We now discuss characterizations of approachability in the case of full monitoring. We will need the stronger
notion of one-shot approachability (the notion of one-shot excludability is stated only for later purposes).
\begin{definition}
\label{def:oneshot-orig}
A set $\cC$  is one-shot $r$--approachable if there exists
$x \in \Delta(\cI)$ such that for all $y \in \Delta(\cJ)$, one has $r(x,y) \in \cC$.
A set $\cC$ is one-shot $r$--excludable if for some $\delta>0$,
the complement of the $\delta$--neighborhood of $\cC$ is one-shot $r$--approachable by Nature.
\end{definition}
\cite{Bla56} (see also \citealt{MeSoZa94}) provided the following \emph{primal characterization} of
approachable convex\footnote{This primal characterization was actually stated by \cite{Bla56} in a more general way
for all, non-necessarily convex, sets.} sets. A set that satisfies it is called a $B$--set.
\begin{theorem}\label{CharacBset}
A convex set $\cC$ is $(r,\Full)$--approachable if and only if any containing half-space
$\Chs \supseteq \cC$ is one-shot $r$--approachable.
\end{theorem}

This characterization also leads to an approachability strategy,
which we describe with a slight modification with respect to its most classical
statement. We denote by $r'_t = r(x_t,j_t)$ the expected payoff obtained at round $t$, which
is a quantity that is observed by the player. At stage $n$, if $\bar{r}'_n \not\in \cC$,
let $\pi_{\cC}(\bar{r}'_n)$
denote the projection of $\bar{r}'_n$ onto $\cC$ and consider the containing half-space
$\Chsind{n+1}$ whose defining hyperplane is tangent to $\cC$ at $\pi_{\cC}(\bar{r}'_n)$.
The strategy then consists of choosing the mixed action $x_{n+1} \in \Delta(\cI)$ associated with the one-shot approachability of $\Chsind{n+1}$, as illustrated in Figure~\ref{FG:Bset}.

\begin{figure}[h!]
\begin{center}
\includegraphics{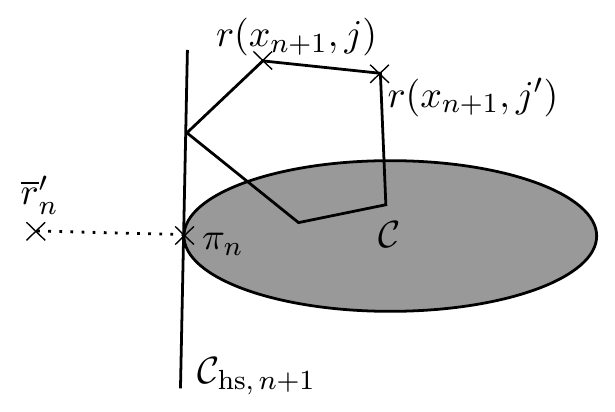}
\caption{\label{FG:Bset} An illustration of Blackwell's approaching strategy. At stage $n+1$, when
$\overline{r}'_n \not\in \cC$, the convex set $\cC$ and the expected payoff $r'_{n+1}$ are in the half-space $\Chsind{n+1}$
while $\overline{r}'_n$ lies in its complement.}
\end{center}
\end{figure}

If $\bar{r}'_n \in \cC$, any choice $x_{n+1} \in \Delta(\cI)$ is suitable.
The above strategy ensures that for all $y \in \Delta(\cJ)$,
\begin{equation}
\label{eq:cvrate1}
\bigl\langle r(x_{n+1},y) - \pi_{\cC}(\bar{r}'_n), \,\, \bar{r}'_n - \pi_{\cC}(\bar{r}'_n) \bigr\rangle \leq 0\,;
\end{equation}
which in turn ensures the convergence to $\cC$ of the mixed payoffs at a rate independent of $d$, namely
\begin{equation}
\label{eq:cvrate2}
d_{\cC}\!\left( \frac{1}{n} \sum_{t=1}^n r(x_t,j_t) \right) \leq \frac{2M}{\sqrt{n}}\,.
\end{equation}
The uniform convergence of $\bar{r}_n$ to $\cC$ is deduced by martingale convergence theorems
(e.g., the Hoeffding--Azuma inequality) from the above
uniform convergence of $\bar{r}'_n$ to $\cC$.
\medskip

\paragraph{Dual characterization.}
In the specific case of closed convex sets, using von Neumann's min-max theorem, the primal
characterization stated above can be transformed into the following \emph{dual characterization}:
\begin{equation}
\label{eq:dualfull}
\cC\subseteq \R^d\ \mbox{is}\ (r,\Full)\mbox{--approachable} \quad \Longleftrightarrow \quad
\forall\, y  \in \Delta(\cJ),\ \exists\, x\in\Delta(\cI), \ \ r(x,y)\in \cC\,.
\end{equation}
This characterization might be simpler to formulate and to check, yet it does not provide an explicit approachability strategy.

\section{Related literature and the objective of this paper}
\label{sec:relatedandobjective}

In this section we first recall the existing results on approachability with partial monitoring
and then explain in a more technical way the objectives of the paper.

\subsection*{Results on approachability with partial monitoring}

\paragraph{Concerning the primal characterization.} \cite{Koh75}---see also \citealt{MeSoZa94}---studied
specific frameworks (induced by repeated games with incomplete information, see Section~\ref{SE:IncompleteInfo})
in which approachability depends mildly on the signaling structure.
A property that we define in the sequel and call the \textsl{upper-right-corner property} holds
between the payoff function $r$ and the signaling structure $H$.
Based on this property it is rather straightforward to show that the primal characterization for the $(r,H)$--approachability
of orthants stated in Theorem~\ref{CharacBset} still holds.
Section~\ref{SE:IncompleteInfo} provides more details on this matter.

\paragraph{Concerning the dual characterization.}
\cite{Per11} provided the following dual characterization of approachable closed convex sets under partial monitoring:
\begin{equation}
\label{EQ:charac}
\cC \subseteq \R^d \ \mbox{is} \ (r,H)\mbox{--approachable} \quad \Longleftrightarrow \quad
\forall\, \bh \in \cF, \ \exists \, x\in\Delta(\cI), \ \om(x,\bh)\subseteq \cC\,.
\end{equation}
It indeed generalizes Blackwell's dual characterization~\eqref{eq:dualfull} with full monitoring,
as $\cF$ can be identified with $\Delta(\cJ)$ in this case.

Based on~\eqref{EQ:charac}, Perchet constructed the first $(r,H)$--approachability strategy of any closed convex set $\cC$;
it was based on calibrated forecasts of the vectors of $\cF$.
Because of this, the per-stage computational complexity of this strategy increases indefinitely
and rates of convergence cannot be inferred. Moreover, the construction of this strategy is unhelpful to infer a generic primal characterization.

\cite{Ext} tackled the issue of complexity and devised an efficient $(r,H)$--approachability strategy for the case
where the target set is some polytope. This strategy has a fixed and bounded per-stage computational complexity. Moreover, its rates of convergence
are independent of $d$: they are of the order of $n^{-1/5}$, where $n$ is the number of stages.

On the other hand, \cite{PerQui11}
unified the setups of approachability with full or partial monitoring and characterized
approachable closed (convex) sets using some lifting to the Wasserstein space of probability measures on
$\Delta(\cI)\times\Delta(\cJ)$.

\subsection*{Objectives and technical content of the paper}

This paper focuses on the primal characterization of approachable closed convex sets
with partial monitoring.
First, note that if a closed convex set is $(r,H)$--approachable, then it is also $(r,\Full)$--approachable, and
therefore, by~\eqref{eq:dualfull}, any containing half-space is necessarily one-shot $r$--approachable.
The question is: When is the latter statement a sufficient condition for $(r,H)$--approachability?
The difficulty, as noted already by~\cite{Per11} and recalled at the beginning of Section~\ref{sec:prodmeasureorthants},
is that since the notions of approachability with full or partial monitoring do not coincide,
it can be that a closed convex set is not $(r,H)$--approachable while every containing half-space is one-shot $r$--approachable.

Some situations where the usual dual characterization is indeed sufficient
are formed first, by the cases when the target set is a half-space (with no condition on the game),
and second, by the cases when the target set is an orthant and the structure $(r,H)$ of the game
satisfies the upper-right-corner property.
This first series of results is detailed in Sections~\ref{sec:halfspaces} and~\ref{sec:prodmeasureorthants}.
Some light is then shed in Section~\ref{SE:IncompleteInfo}
on the construction of~\cite{Koh75} for the case of repeated games with incomplete information.

The rest of the paper (Sections~\ref{sec:noproduct}, \ref{SE:Polyhedra}, and~\ref{sec:gen})
relies on no assumption on the structure $(r,H)$ of the game.
It discusses a primal condition based on one-shot approachability of half-spaces
with respect to a modified payoff function $\tr_H$ that encompasses
the links between the signaling structure $H$ and the original payoff function $r$.
Depending on the geometry of the target closed convex set, this
primal condition is stated in the original payoff space (for orthants,
Section~\ref{sec:noproduct}) or in some lifted space (for polytopes or general convex sets, see
Section~\ref{SE:Polyhedra} and~\ref{sec:gen}). We explain in Example~\ref{ex:hidden}
why such a lifting seems inevitable.

We also illustrate how the exhibited primal condition leads to a new (and efficient) strategy for $(r,H)$--approachability
in the case of target sets given by polytopes. (Section~\ref{sec:noproduct} does it for orthants and
the result extends to polytopes via Lemma~\ref{LM:equivalent}.)
This new strategy is based on sequences of (modified) payoffs, as in~\cite{Koh75}, and is not only based
on sequences of signals, as in \cite{Per11,Ext}. The construction of this strategy
also entails some non-linear approachability results (both in full and partial monitoring).

\section{Primal approachability of half-spaces}
\label{sec:halfspaces}

We first focus on half-spaces, not only because they are the simplest convex sets, but because they are the cornerstones of
the primal characterization of \cite{Bla56}.
The following proposition ties one-shot $r$--approachability with $(r,H)$--approachability of half-spaces.

\begin{proposition}
\label{PR:Halfspace}
For all half-spaces $\Chs$, for all signaling structures $H$,
\[
\Chs \ \mbox{is}\  (r,H)\mbox{--approachable} \ \Longleftrightarrow \ \Chs \ \mbox{is one-shot}\ r\mbox{--approachable}\,.
\]
\end{proposition}

This result is a mere interpolation of two well-known results for the extremal cases where $H = \Full$ and $H = \Dark$.
The former case corresponds to Blackwell's primal characterization. In the latter case, Nature could always
play the same $y \in \Delta(\cJ)$ at all rounds and the player cannot infer the value of this $y$. So he needs to have an action $x \in \Delta(\cI)$ such that $r(x,y)$ belongs to $\Chs$, no matter $y$.

Stated differently, the above proposition indicates that as far as half-spaces are concerned, the approachability
is independent of the signaling structure.

\begin{proof}
Only the direct implication is to be proven, as the converse implication is immediate by the above discussion
about games in the dark. We thus assume that $\Chs$ is $(r,H)$--approachable. Using the characterization~\eqref{EQ:charac}
of $(r,H)$--approachable sets, one then has that
\[ \forall \bh \in \cF,\ \exists \,x \in \Delta(\cI), \ \ \om(x,\bh) \subset \Chs\,, \]
which implies that \[ \forall\, y \in \Delta(\cJ),\ \exists\, x \in \Delta(\cI), \ \ r(x,y) \in \Chs\,.\]
The implication holds because $r(x,y) \in \om(x,\bh)$ when $\bh=\bH(y)$.
Now, let $a \in \R^d$ and $b\in \R$ such that $\Chs = \bigl\{ \omega  \in \R^d : \ \langle \omega, a \rangle \leq b \bigr\}$.
The above property can be further restated as
\[ \forall\, y \in \Delta(\cJ),\ \exists\, x \in \Delta(\cI),\ \ \langle r(x,y),a \rangle \leq b\,,\]
or equivalently, \[\max_{y \in \Delta(\cJ)} \min_{x \in \Delta(\cI)} \langle r(x,y),a \rangle \leq b\,.\]

By von Neumann's min-max theorem, we then have that
\[
\min_{x \in \Delta(\cI)} \max_{y \in \Delta(\cJ)} \langle r(x,y),a \rangle \leq b\,, \quad
\mbox{that is,} \quad \exists\, x_0 \in \Delta(\cI), \ \forall y \in \Delta(\cJ),\ \ \langle r(x_0,y),a \rangle \leq b\,.
\]
This is exactly the one-shot approachability of $\Chs$.
\end{proof}

Since the complement of any $\delta$--neighborhood of a half-space is also a half-space, we get the following additional
equivalence, in view of the respective definitions of excludability and one-shot excludability.

\begin{corollary}
For all half-spaces $\Chs$, for all signaling structures $H$,
\[
\Chs \ \mbox{is}\  (r,H)\mbox{--excludable} \ \Longleftrightarrow \ \Chs \ \mbox{is one-shot} \ r\mbox{--excludable}\,.
\]
\end{corollary}

\section{Primal approachability of orthants \\ under the upper-right-corner property}
\label{sec:prodmeasureorthants}

This section is devoted to stating a primal characterization of $(r,H)$--approachable orthants, i.e., sets of the form
\[
\Corth(a) = \bigl\{ a - \omega : \ \omega \in (\R_+)^d \bigr\}
\]
for some $ a \in \R^d$.
Orthants are the key for extension to polytopes, because, as we will discuss later, up to some lifting in higher dimensions,
every polytope can be seen as an orthant.

We start by indicating that the primal characterization stated in the previous section in terms
of the original payoff function $r$ does not
extend directly to general convex sets, not even to orthants---at least without an additional assumption. However, in this section we
state such a sufficient assumption for its extension. We study the most general primal characterization
in Section~\ref{sec:noproduct}, which will involve a modified payoff function for the one-shot approachability
of half-spaces.
\medskip

\paragraph{Counter-example (adapted from~\citealp{Per11}).}
We  show that the equivalence of Proposition~\ref{PR:Halfspace} does not hold in general
if the convex set $\cC$ at hand is not a half-space. To do so, we exhibit a game and a set $\cC$ which is $(r,\Full)$--approachable but not $(r,\Dark)$--approachable.
We set $\cI=\{T,B\}$ and $\cJ=\{L,R\}$, and the payoff function $r$ is given by the matrix
\[
\begin{array}{cccc}
\hline
& & L & R \\
\hline
T & & (0,0) & (1,-1) \\
B & & (-1,1) & (0,0) \\
\hline
\end{array}
\]
We consider the set $\Corth\bigl((0,0)\bigr) = (\R_-)^2$.
This set is $(r,\Full)$--approachable as indicated by the dual characterization~\eqref{eq:dualfull}:
for each $\alpha \in [0,1]$,
\[
r\bigl(\alpha T+(1-\alpha)B,\,\alpha L+(1-\alpha)R\bigr)=(0,0) \in \Corth\bigl((0,0)\bigr)\,.
\]
On the other hand, consider the signaling structure $H = \Dark$, for which the only flag is $\emptyset$.
For all actions of the player, i.e., for all $\alpha \in [0,1]$, it holds that
\begin{multline*}
\om\bigl(\alpha T+(1-\alpha)B,\,\,\emptyset\bigr) =
\Bigl\{ r\bigl(\alpha T+(1-\alpha)B,\,\,y\bigr) : \ y \in \Delta(\cJ) \Bigr\} \\
= \bigl\{ (\lambda,-\lambda)\,;\ \lambda \in [-\alpha,\,1-\alpha] \bigr\} \nsubseteq \Corth\bigl((0,0)\bigr)\,.
\end{multline*}
Therefore, the characterization of $r$--approachable closed convex sets~(\ref{EQ:charac}) does not hold when
playing in the dark.
\medskip

\paragraph{Upper-right-corner property.}
We define the \textsl{upper-right corner function} $R: \Delta(\cI) \times \cF \to \R^d$
of the compatible payoff function $\om$ in a component-wise manner. We write
the coordinates of $R$ as $R = (R_1,\ldots,R_d)$ and set, for all $k \in \{1,\ldots,d\}$,
\[
\forall\, x \in \Delta(\cI), \ \forall\, \bh \in \cF, \qquad
R_k(x,\bh) = \max \Big\{ \omega^k : \ \omega = \bigl(\omega^1,\ldots,\omega^d\bigr) \in \om(x,\bh) \Big\}\,.
\]
The construction of $R$ is illustrated in Figure~\ref{FIG:R}.
Note that the $\ell_2$--norm of $R$ is in general bounded by $M \sqrt{d}$.

\begin{figure}[h!]
\begin{center}
\includegraphics{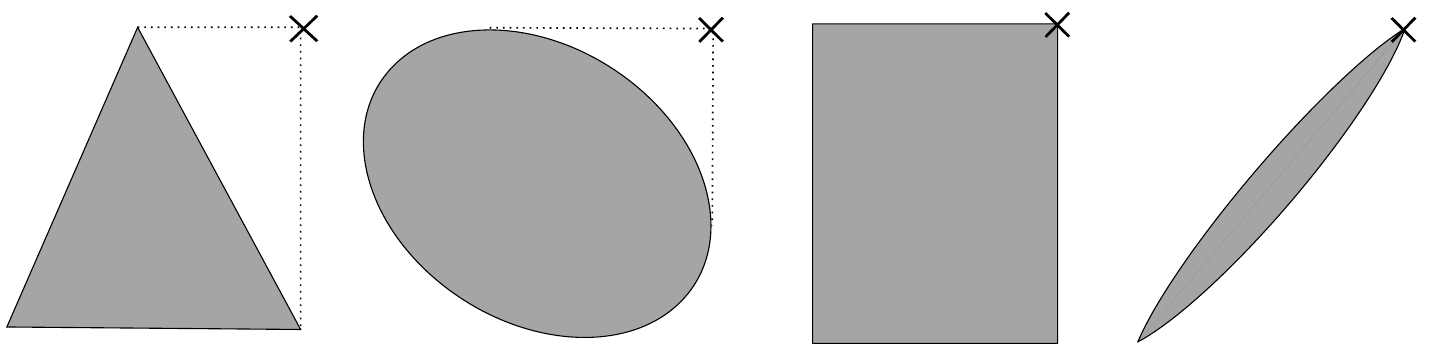}
\caption{\label{FIG:R} Four illustrations of compatible payoff sets $\om(x,\bh)$ and associated upper-right corners $R(x,\bh)$. In the two examples on the left, this upper-right corner does not belong to the set, while it does in the two on the right. (When it is so for all $x$ and $\bh$,
the game is said to have the upper-right-corner property.) }
\end{center}
\end{figure}

The term ``upper-right corner'' comes from the fact that $R(x,\bh)$ is the (component-wise) smallest
$a$ such that $\om(x,\bh) \subseteq \Corth(a)$. Controlling
the distance of $R(x,\bh)$ to the orthant entails controlling the distance of the whole
set $\om(x,\bh)$ to it. Thus, the point $R(x,\bh)$
is in some sense the worst-case payoff vector associated with $\om(x,\bh)$.

Of course, $R(x,\bh)$ is in general not a feasible payoff vector, i.e., $R(x,\bh) \not\in \om(x,\bh)$. We are interested in this section in the case where the upper-right corner is indeed
a feasible payoff---an assumption that we call the \textsl{upper-right-corner property}.
\begin{definition}
\label{def:productproperty}
The game $(r,H)$ with partial monitoring has the \emph{upper-right-corner property} if
\[
\forall\, x \in \Delta(\cI), \ \forall\, \bh \in \cF, \quad R(x,\bh) \in \om(x,\bh)\,.
\]
\end{definition}

Of course, games with full monitoring have the upper-right-corner property, as
for them $\cF$ can be identified with $\Delta(\cJ)$ and $\om$ can be identified with
the function $\{ r \}$
with values in the set of all singleton subsets of $\R^d$.

By definition, in a game with the upper-right-corner property, the $\ell_2$--norm of $R$ is bounded by $M$.

\paragraph{Primal characterization under the upper-right-corner property.}
The following proposition was implicitly used by \cite{Koh75}.
\begin{proposition}\label{PR:Kohlberg}
For all games $(r,H)$ with partial monitoring that have the upper-right-corner property,
for all orthants $\Corth(a)$, where $a \in \R^d$,
\begin{align*}
& \Corth(a) \ \mbox{is}\  (r,H)\mbox{--approachable} \\
\Longleftrightarrow \quad  & \mbox{every half-space $\Chs \supset \Corth(a)$
is one-shot} \  r\mbox{--approachable}.
\end{align*}
\end{proposition}

Stated differently, using Blackwell's primal characterization of approachability (Theorem~\ref{CharacBset}),
an orthant $\Corth(a)$ is $(r,H)$--approachable in a game $(r,H)$ satisfying the upper-right-corner property
if and only if $\Corth(a)$ is $r$--approachable with full monitoring.

\begin{proof}
The direct implication is proved by applying Proposition~\ref{PR:Halfspace} to any half-space $\Chs \supset \Corth(a)$,
which is in particular $(r,H)$--approachable as soon as $\Corth(a)$ is. The interesting implication is thus the converse one.

So, we assume that every half-space $\Chs \supset \Corth(a)$ is one-shot $r$--approachable and,
following Kohlberg's original proof and inspired by Blackwell's strategy in the case of full monitoring,
we construct an $(r,H)$--approachability strategy of $\Corth(a)$. \medskip

\emph{Flags observed, mixed payoffs obtained.~~}For simplicity, assume first
that after stage $n$, the observation made by the player is not just the random signal $s_n$
but the entire vector of probability distributions over the signals $\bh_n = \bH(y_n)$. (We will indicate below
why this is not a restrictive assumption.) We consider the surrogate payoff vector
$R_n = R(x_n,\bh_n)$, which is a quantity thus observed by the player.
When $\overline{R}_n$ does not already belong to $\Corth(a)$, since the latter set
is convex, the half-space $\Chsind{n}$ defined by
\[
\Chsind{n} = \left\{\omega \in \R^d : \ \   \left\langle\omega-\pi_{\Corth(a)}\bigl(\overline{R}_n\bigr),\,\,\overline{R}_n-\pi_{\Corth(a)}\bigl(\overline{R}_n\bigr)\right\rangle
\leq 0 \right\}
\]
contains $\Corth(a)$, where we recall that $\pi_{\Corth(a)}$ is the orthogonal projection onto $\Corth(a)$.
By assumption, $\Chsind{n}$ is thus one-shot $r$--approachable.
That is, there exists $x_{n+1} \in \Delta(\cI)$ such that
\begin{equation}
\label{eq:ppy-Fig3}
\forall\, y \in \Delta(\cJ), \qquad
\left\langle r(x_{n+1},y)-\pi_{\Corth(a)}\bigl(\overline{R}_n\bigr),\,\,\overline{R}_n-\pi_{\Corth(a)}\bigl(\overline{R}_n\bigr)\right\rangle
\leq 0\,.
\end{equation}
By the upper-right-corner property, $R_{n+1} \in \om(x_{n+1},\bh_{n+1})$, which entails that
there exists $y'_{n+1} \in \Delta(\cJ)$ such that $\bH\bigl(y'_{n+1}\bigr) = \bh_{n+1}$
and $R_{n+1} = r\bigl(x_{n+1},y'_{n+1}\bigr)$. As a consequence, $R_{n+1}$ belongs to $\Chsind{n}$
and the sequence $(R_n)$ satisfies the following condition,
usually referred to as \textsl{Blackwell's condition}:
\[
\left\langle R_{n+1}-\pi_{\Corth(a)}\bigl(\overline{R}_n\bigr),\,\,\overline{R}_n-\pi_{\Corth(a)}\bigl(\overline{R}_n\bigr)\right\rangle
\leq 0\,.
\]
This condition is trivially satisfied when $\overline{R}_n$ already belongs to $\Corth(a)$.
Just as~\eqref{eq:cvrate1} leads to~\eqref{eq:cvrate2}, this condition implies that
$d_{\Corth(a)}\bigl(\overline{R}_n\bigr) \leq 2M/\sqrt{n}$.

Now, $(1/n) \sum_{t=1}^n r(x_t,y_t) \in (1/n) \sum_{t=1}^n \om(x_t,\bh_t)$ and,
as $R$ is the upper-right corner function,
$(1/n) \sum_{t=1}^n \om(x_t,\bh_t) \subseteq \Corth\bigl(\overline{R}_n\bigr)$. That is,
$(1/n) \sum_{t=1}^n r(x_t,y_t)$ is component-wise smaller than $\overline{R}_n$.
Since the distance to the orthant $\Corth(a)$ equals, for all $\omega \in \R^d$,
\begin{equation}
\label{eq:distCorth}
d_{\Corth(a)}(\omega) = \sqrt{ \sum_{k=1}^d \max\{\omega_k-a_k,\,0\}^2 }\,,
\end{equation}
we get that
\[
d_{\Corth(a)}\!\left( \frac{1}{n} \sum_{t=1}^n r(x_t,y_t) \right) \leq
d_{\Corth(a)}\bigl(\overline{R}_n\bigr) \leq \frac{2M}{\sqrt{n}}\,.
\]
Finally, by martingale convergence theorems (e.g., the Hoeffding--Azuma inequality),
the sequence of the distances $d_{\Corth(a)}\bigl(\overline{r}_n\bigr)$ converges uniformly to $0$
with respect to strategies of Nature. The above arguments are illustrated in Figure~\ref{FIG:Orth}. \medskip

\begin{figure}[h!]
\begin{center}
\includegraphics{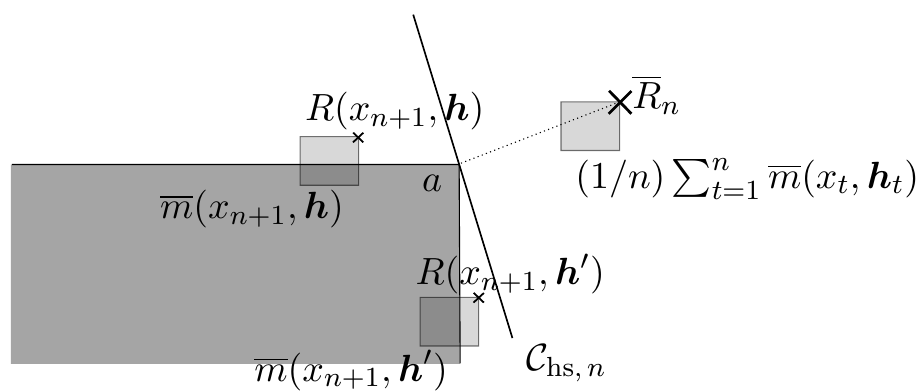}
\caption{\label{FIG:Orth} Illustration of the guarantee~\eqref{eq:ppy-Fig3}, in the case when the sets
of compatible payoffs $\om(x,\bh)$ are all given by rectangles.}
\end{center}
\end{figure}

\emph{Flags not observed, only random signals observed, pure payoffs.~~}
It remains to relax the assumption that the flags $\bh_n$ are observed, while only the signals $s_n$ drawn at random according to
$H(i_n,j_n)$ are. A standard trick in the literature of partial monitoring  (see \citealt{Koh75}, \citealt{MeSoZa94}, \citealt{LMS08})
solves the issue, together with martingale convergence theorems and the fact that the upper-right
corner $R$ is a Lipschitz function for a well-chosen metric over sets (see Lemma~\ref{LM:tr} below). We briefly describe
this trick without working out the lengthy details.
Time is divided into blocks of time indexed by $b=1,2,\ldots$ and with respective (large and increasing) lengths $L_b$.
Another sequence of elements $\gamma_b \in (0,1)$ converging to $0$ is needed.
The same mixed distribution $x^{(b)}$ is played at all stages of block $b$ by the player; that is, $x_{bL+t} = x^{(b)}$ for all $1 \leq t \leq L_b$.
This distribution is chosen by mixing a distribution $x_{\mbox{\tiny orig.}}^{(b)}$ satisfying a constraint of the form of~\eqref{eq:ppy-Fig3}
with the uniform distribution. This is done with respective weights $1-\gamma_b$ and $\gamma_b$.
The distribution $x^{(b)}$ then puts a positive probability mass of at least $\gamma_b > 0$ on all actions. Doing so,
an estimator of the average flag on block $b$ can be constructed. Its accuracy,
as well as the price to pay for the mixing, depend on $\gamma_b$ and $L_b$. By such a price,
we mean how farther away we are from $\Corth(a)$ because we did not play $x_{\mbox{\tiny orig.}}^{(b)}$
but $x^{(b)}$ instead.
Informally, each block now plays the role of a stage in the setting above when flags were observed.
One can show that suitable values of $L_b$ and $\gamma_b$ lead to uniform convergence of the average
payoffs, measured in terms of pure actions $x_t$, to $\Corth(a)$.
Also, similar martingale convergence arguments show that measuring payoffs in terms of the mixed
actions $x_t$ or pure actions $i_t$ does not matter.

As indicated, we omit the technical proof of these facts (it already appeared in all the given references)
but notice, however, that rates of convergence are adversely affected by this trick.
\end{proof}

\begin{remark}{}
The construction in the proof above shows that under the upper-right-corner property, it is necessary and sufficient
to control the behavior of the upper-right corners~$R_n$.
\label{RM:KohlbergLemma}

This property was used only to show that $R_{n+1}$ is equal to some $r(x_{n+1},y'_{n+1})$
and thus that the sequence $(R_n)$ satisfies Blackwell's condition.
When the property is not satisfied anymore, the sequence of the upper-right corners may fail to satisfy this condition.
For instance, in the counter-example at the beginning of the present section, the upper-right corners equal
\[
\forall\, \alpha \in [0,1], \qquad
R\bigl(\alpha T+(1-\alpha)B,\,\,\emptyset\big) = (\alpha,1-\alpha)\,,
\]
so that, for all strategies of the player, $\overline{R}_n = (\lambda,1-\lambda)$ for some $\lambda \in [0,1]$. Thus,
the distance of $\overline{R}_n$ to $\Corth\bigl((0,0)\bigr)$ is always larger than $1/\sqrt{2}$.
\end{remark}

\section{Intermezzo: \\ Kohlberg's repeated games with incomplete information}
\label{SE:IncompleteInfo}

We consider in this section a different, yet related framework,
which is the main focus of~\citet{Koh75}. We first describe
a setting where $d$ games with partial monitoring are to be
played simultaneously, and then establish the formal connection
with Kohlberg's results.

\paragraph{Simultaneous games with partial monitoring.}
We consider $d$ such games, with common action sets $\cI$ and $\cJ$
for the player and Nature and common set $\cH$ of signals, but with possibly different
payoff functions and signaling structures. We index these games by $g$. For each game
$g \in \{1,\ldots,d\}$, the player's payoff function is denoted by
$r^{(g)} : \cI \times \cJ \to \R$ and the signaling structure is given by $H^{(g)} : \cI \times \cJ \to \Delta(\cH)$,
with associated maximal information mapping $\bH^{(g)} : \Delta(\cJ) \to \Delta(\cH)^{\cI}$.

We put some restrictions on the strategies of the player and of Nature.
The player may only choose one action $x_t \in \Delta(\cI)$ at each stage $t$, the same for all games $g$.
On the other hand, Nature can choose different mixed actions $y^{(g)}_t$ in each game $g$,
but these need to be \emph{non-revealing}, that is, they need to induce the same flags. More
formally, they need to be picked in the following set, which we assume to be non-empty:
\begin{equation}
\label{eq:NR}
\NR = \Bigl\{ \bigl( y^{(1)}, \ldots, y^{(d)} \bigl) \in \Delta(\cJ)^d : \ \ \bH^{(1)}\bigl(y^{(1)}\bigr) = \cdots =
\bH^{(d)}\bigl(y^{(d)}\bigr) \Bigr\}\,.
\end{equation}
\medskip

The above framework of simultaneous games can be embedded into an equivalent game that fits the model
studied in the previous sections. Indeed, by linearity of each $\bH^{(g)}$, the set $\NR$ of non-revealing actions
is a polytope, thus it is the convex hull of its finite set of extremal points. We denote the cardinality of
the latter by $K$ and we write its elements as
\[
\cK = \Bigl\{
\bigl( y^{(g)}_1 \bigr)_{1 \leq g \leq d}, \,\, \ldots, \,\, \bigl( y^{(g)}_K \bigr)_{1 \leq g \leq d}
\Bigr\}.
\]
Each $\bigl( y^{(g)} \bigr)_{1 \leq g \leq d} \in \NR$ can then be represented by an element of $\Delta(\cK)$.
Conversely, each $z = (z_k)_{k \leq K} \in \Delta(\cK)$ induces the following element of $\NR$:
\[
Y(z) = \bigl( Y^{(g)}(z) \bigr)_{1 \leq g \leq d} = \sum_{k=1}^K z_k\,\bigl( y^{(g)}_k \bigr)_{1 \leq g \leq d}\,.
\]
So, with no loss of generality, we can assume that $\cK$ is the finite set of actions of Nature and that, given
$z \in \Delta(\cK)$ and $x \in \Delta(\cI)$, the payoff in the game $g$ is $r^{(g)}\bigl(x,Y^{(g)}(z)\bigr)$.

This defines naturally an auxiliary game with linear vector-valued payoff function $r : \Delta(\cI) \times \Delta(\cK) \to \R^d$
and maximal information mapping $\bH : \Delta(\cK) \to \Delta(\cH)^{\cI}$ defined by
\[
r(x,z) = \Bigl( r^{(g)}\bigl(x,Y^{(g)}(z)\bigr) \Bigr)_{1 \leq g \leq d} \qquad \mbox{and} \qquad
\bH(z) = \bH^{(g)}\bigl( Y^{(g)}(z) \bigr) \quad \mbox{for all $g$}.
\]
(The definition of $\bH$ is independent of $g$ by construction, as we restricted Nature to use non-revealing strategies.)
This maximal information mapping $\bH$ corresponds to an underlying signaling structure which we denote by $H : \cI \times \cK \to \Delta(\cH)$.
\medskip

The game $(r,H)$ constructed above satisfies the upper-right-corner property. Indeed, for all $\bh \in \bH\bigl(\Delta(\cK)\bigr)$
and all $x \in \Delta(\cI)$,
\begin{align*}
\om(x,\bh) & = \biggl\{
\Bigl( r^{(1)}\bigl(x,Y^{(1)}(z)\bigr), \, \ldots, \, r^{(d)}\bigl(x,Y^{(d)}(z)\bigr) \Bigr) :
\ \ z \in \Delta(\cK) \ \mbox{s.t.} \ \bH(z)=\bh \biggr\} \\
& = \biggl\{
\Bigl( r^{(1)}\bigl(x,y^{(1)}\bigr), \, \ldots, \, r^{(d)}\bigl(x,y^{(d)}\bigr) \Bigr) : \\
& \qquad \qquad \bigl( y^{(1)},\ldots,y^{(d)} \bigr) \in \Delta(\cJ)^d
\ \mbox{s.t.} \ \bH^{(1)}\bigl( y^{(1)} \bigr) = \cdots = \bH^{(d)}\bigl( y^{(d)} \bigr)
= \bh \biggr\}.
\end{align*}
Because of the separation of the variables in the constraint, the following set, given $\bh$,
\[
\Bigl\{
\bigl( y^{(1)},\ldots,y^{(d)} \bigr) \in \Delta(\cJ)^d :
\ \bH^{(1)}\bigl( y^{(1)} \bigr) = \cdots = \bH^{(d)}\bigl( y^{(d)} \bigr)
= \bh \Bigr\}
\]
is a Cartesian product of subsets of $\Delta(\cJ)$. Thus, its image $\om(x,\bh)$ by the mapping
\[
\bigl( y^{(1)},\ldots,y^{(d)} \bigr) \in \Delta(\cJ)^d \,\, \longmapsto \,\,
\Bigl( r^{(1)}\bigl(x,y^{(1)}\bigr), \, \ldots, \, r^{(d)}\bigl(x,y^{(d)}\bigr) \Bigr)
\]
is also a Cartesian product of closed intervals of $\R$.
In particular, the latter set contains its upper-right corner,
that is, $R(x,\bh) \in \om(x,\bh)$, as claimed.
\medskip

We assume, with no loss of generality, that in these simultaneous games,
Nature maximizes the payoffs and the player minimizes them.
A question that naturally arises---and whose answer will be needed below---is to
determine for which vectors $a = (a_1,\ldots,a_d) \in \R^d$
the player can simultaneously guarantee that his average payoff will be smaller than $a_g$
in the limit in each game $g \in \{1,\ldots,d\}$; that is, to determine
which orthants $\Corth(a)$ are $(r,H)$--approachable.
By the exhibited upper-right-corner property, Proposition~\ref{PR:Kohlberg} shows that
a necessary and sufficient condition for this is that all containing
half-spaces of $\Corth(a)$ be one-shot $r$--approachable.
These half-spaces are parameterized by the convex distributions
$q \in \Delta\bigl(\{1,\ldots,d\}\bigr)$ and are denoted by
\begin{equation}
\label{eq:Kohlaux}
\Chs^{(q)} = \big\{\omega \in \R^d : \ \langle \omega,q \rangle \leq \langle a,q \rangle \bigr\}\,.
\end{equation}
Stated equivalently, the orthant $\Corth(a)$
is $(r,H)$--approachable if and only if the value of the zero-sum game with payoff function
$(x,z) \in \Delta(\cI) \times \Delta(\cK) \mapsto
\langle r(x,z),\,q \rangle$ is smaller than $\langle a, q \rangle$ for all $q \in \Delta(\{1,\ldots,d\})$.
\bigskip

\paragraph{Kohlberg's model of repeated games with incomplete information.}
The setting of repeated games with incomplete information, introduced by \cite{AuMa55},
relies on the same finite family of games $\bigl( r^{(g)}, \, H^{(g)} \bigr)$, where
$g \in \{1,\ldots,d\}$, as described above. They will however not be played simultaneously.
Instead, a single game (state) $G \in \{1,\ldots,d\}$ is drawn according to some
probability distribution $p \in \Delta\bigl(\{1,\ldots,d\}\bigr)$
known by both the player and Nature. Yet only Nature (and not the player) is informed of the true state $G$.
A repeated game with partial monitoring then takes place
between the player and Nature in the $G$--th game. Payoffs are evaluated in expectation
with respect to the random draw of $G$ according to $p$.

For simplicity, we assume that all mappings $\bH^{(g)}$
have the same range\footnote{In full generality, when this is not the case,
Nature may resort to strategies that reveal that the true state $G$ belongs to some strict subset of $\{1,\ldots,d\}$, and
the player must adapt his strategy in correspondence with this knowledge, see~\cite{Koh75}.
But our assumption already captures the basic idea of the use of approachability in this framework and
the alluded technical adaptations are beyond the scope of this paper.} and, with no loss of generality, that $p$ has full support.
Because of these two properties, the considered setting
of repeated games with incomplete information can then be embedded, from the player's viewpoint, into the above-described
setting of $d$ simultaneous games under the restriction that Nature resorts to non-revealing strategies.
Indeed, from the player's viewpoint and because of the identical range of the $\bH^{(g)}$,
the mixed action used by Nature in the game $G$ can be thought of as the $G$--th component of
some vector of mixed actions in the set $\NR$ defined in~\eqref{eq:NR}. We use the notation defined above:
as payoffs are evaluated in expectation,
the payoff function is formed by the inner products
$(x,z) \in \Delta(\cI) \times \Delta(\cK) \mapsto \langle r(x,z),\,p \rangle$.
\medskip

We recall that Nature maximizes the payoff and that the player minimizes it.
For each $q \in \Delta(\{1,\ldots,d\})$, we denote by $u(q)$ the value of the one-shot zero-sum game $\Gamma(q)$ with payoffs
$(x,z) \in \Delta(\cI) \times \Delta(\cK) \mapsto \langle r(x,z),\,q \rangle$.
We show that, as proved in the mentioned references,
the value $U$ of this repeated game, as a function of the distribution $p$, may be larger than $u(p)$ and is
given by $\cav[u](p)$, where $\cav[u]$ is the smallest concave function above~$u$.

First, the so-called splitting lemma shows that $U$ is concave. Therefore, we have $U \geq \cav[u]$.
(For the splitting lemma, see~\citealt{AuMa55} and also~\citealt[Section~V.1]{MeSoZa94} or \citealt{SorZS}.)
The inequality of interest to us is the converse one. Using the concavity of the mapping $\cav[u]$,
\citet[Corollary~2.4]{Koh75} proves that for all $p \in \Delta(\{1,\ldots,d\})$,
there exists some $a_p \in \R^d$ such that
\[
\cav[u](p) = \langle a_p, p \rangle \qquad \mbox{and} \qquad
\forall \, q \in \Delta(\{1,\ldots,d\}), \quad \cav[u](q) \leq \langle a_p, q \rangle\,.
\]
In particular, $u(q) \leq \langle a_p, q \rangle$ for all $q \in \Delta(\{1,\ldots,d\})$.
The equivalence stated after~\eqref{eq:Kohlaux}
shows that $\Corth(a_p)$ is therefore $(r,H)$--approachable. Hence, no matter the strategy of Nature, the payoff in state $G$
is asymptotically smaller than the $G$--th component of $a_p$. (This is true for all realizations of $G$.)
As a consequence, in expectation (with respect to the random choice of $G$), the payoff is smaller than $\langle a_p,\,p \rangle = \cav[u](p)$.
This shows that $U(p) \leq \cav[u](p)$.
\medskip

In conclusion,
\cite{Koh75} implicitly used the consequences of the upper-right-corner property detailed above
when constructing an optimal strategy for the uninformed player.
A close inspection reveals that Lemma~5.4 therein does not hold anymore in the more general framework
without the upper-right-corner property (in particular, one might want to read it again with Remark~\ref{RM:KohlbergLemma} in mind).

\section{Primal approachability of orthants in the general case}
\label{sec:noproduct}

We noted that the primal characterization in terms of one-shot $r$--approachability of
containing half-spaces stated in Proposition~\ref{PR:Kohlberg}
did not extend to games $(r,H)$ without the upper-right-corner property.
We show in this section that it holds true in the general case when one-shot approachability is with respect
to the modified payoff function $\tr_H : \Delta(\cI)\times\Delta(\cJ) \to \R^d$ defined as follows:
\[
\forall\, x \in \Delta(\cI), \ \forall\, y \in \Delta(\cJ), \qquad  \tr_H(x,y) = R\bigl(x,\bH(y)\bigr)\,.
\]
The change of payoff function can be intuitively explained as follows.
As noted in Section~\ref{sec:prodmeasureorthants},
when the target sets are given by orthants (and only because of this),
the behavior of (averages of) sets of compatible payoffs is dictated by their upper-right corners.
Now, the upper-right-corner property indicated that even when measuring payoffs with $r$, the worst-case
payoffs were given by the upper-right corners and that it was thus necessary and sufficient to consider the latter.
If this property does not hold, then evaluating actions with $\tr_H$ enables and forces the consideration of these corners.

Of course, in the case of full monitoring, as follows from the comments
after Definition~\ref{def:productproperty}, no modification takes
place in the payoff function, that is, $\tr_{\Full} = r$.
\bigskip

The main result of this section is the following primal characterization. The rest of this section
will then show how it leads to a new approachability strategy under partial monitoring, based on surrogate payoffs (upper-right corner payoffs)
and not only on signals or on estimated flags, as previously done in the literature (e.g., in the references mentioned in the last part of the proof of Proposition~\ref{PR:Kohlberg}).

\begin{theorem}
\label{TH:orthant}
For all games $(r,H)$ with partial monitoring, for all orthants $\Corth(a)$, where $a \in \R^d$,
\begin{align*}
& \Corth(a) \ \mbox{is} \ (r,H)\mbox{--approachable} \\
\Longleftrightarrow \quad  & \mbox{every half-space $\Chs \supset \Corth(a)$
is one-shot} \ \tr_H\mbox{--approachable}.
\end{align*}
\end{theorem}

The proof of this theorem is as follows.
The dual characterization~\eqref{EQ:charac} indicates that a necessary and sufficient condition of
$(r,H)$--approachability for $\Corth(a)$ is that for all $y \in \Delta(\cJ)$, there exists $x \in \Delta(\cI)$
such that $\om\bigl(x,\bH(y)\bigr) \subseteq \Corth(a)$. Since, by construction of $R$,
the smallest orthant (in the sense of inclusion) in which $\om\bigl(x,\bH(y)\bigr)$ is contained is precisely
$\Corth\bigl( \tr_H(x,y) \bigr)$, the necessary and sufficient condition can be restated as the requirement that
for all $y \in \Delta(\cJ)$, there exists $x \in \Delta(\cI)$ such that $\tr_H(x,y) \in \Corth(a)$.
Now, this reformulated dual characterization of approachability in the context of orthants
is seen to be equivalent to the following primal characterization, which concludes the proof of the theorem.

\begin{proposition}
\label{PR:NonLinearApproach1}
For all games $(r,H)$ with partial monitoring, for all orthants $\Corth(a)$, where $a \in \R^d$,
\begin{align*}
& \forall\, y \in \Delta(\cJ), \ \exists\, x \in \Delta(\cI), \quad \tr_H(x,y) \in \Corth(a) \\
\Longleftrightarrow \quad  & \mbox{every half-space $\Chs \supset \Corth(a)$
is one-shot} \ \tr_H\mbox{--approachable}.
\end{align*}
\end{proposition}

Before proving this proposition, we need to state some properties of the function $\tr_H$.
Given two points $a, a' \in \R^d$, the notation $a \preccurlyeq a'$ means that $a$ is component-wise smaller than $a'$---or equivalently, that
$a$ belongs to the orthant $\Corth(a')$.

\begin{lemma}
\label{LM:tr}
The function $\tr_H$ is Lipschitz continuous. It is also convex in its first argument
and concave in its second argument, in the sense that, for all $x,x' \in \Delta(\cI)$,
all $y,y' \in \Delta(\cJ)$, and all $\lambda \in [0,1]$,
\begin{align*}
\tr_H\bigl(\lambda x +(1-\lambda)x',\,y\bigr) & \preccurlyeq \lambda \, \tr_H(x,y) + (1-\lambda) \, \tr_H(x',y) \\
\mbox{and} \qquad \qquad \lambda \, \tr_H(x,y) + (1-\lambda) \, \tr_H(x,y') & \preccurlyeq \tr_H\bigl(x,\,\lambda y + (1-\lambda)y'\bigr)\,.
\end{align*}
\end{lemma}

\begin{proof}
Convexity and concavity follow from the concavity and the convexity of $\om$ for inclusion. Formally,
it follows from the very definition~\eqref{def:om} of $\om$ and from the linearity of $r$ and $\bH$ that,
for all $x,x' \in \Delta(\cI)$, all $\bh, \bh' \in \cF$, and $\lambda \in [0,1]$,
\begin{align*}
\om\bigl(\lambda x +(1-\lambda)x', \,\bh\bigr) & \subseteq \lambda\,\om(x,\bh) + (1-\lambda)\,\om(x',\bh) \\
\mbox{and} \qquad \qquad \lambda\,\om(x,\bh) + (1-\lambda)\,\om(x,\bh') & \subseteq \om\bigl(x,\, \lambda\bh+(1-\lambda)\bh' \bigr) \,.
\end{align*}
The second part of the lemma follows by taking upper-right corners, which
is a linear and non-decreasing operation (for the respective partial orders $\subseteq$ and $\preccurlyeq$).

As for the Lipschitz property of $\tr_H$, it follows from a rewriting of $\om\bigl(x,\bH(y)\bigr)$ as
\[
\om\bigl(x,\bH(y)\bigr) = \sum_{b \in \mathcal{B}} \phi_b\bigl( \bH(y) \bigr) \, r\bigl( x, \, \bH^{-1}(b) \bigr)\,,
\]
where $\mathcal{B}$ is a finite subset of $\cF$, the $\phi_b$ are Lipschitz functions $\cF \to [0,1]$, and
$\bH^{-1}$ is the pre-image function of $\bH$, which takes values in the set of compact subsets of $\Delta(\cJ)$.
This rewriting was proved in~\citet[Lemma~6.1 and Remark~6.1]{Ext}.
We equip the set of compact subsets of the Euclidian ball with center $(0,\ldots0)$ and radius $M$,
in which $\om$ takes its values, with the Hausdorff distance.
For this distance, $x \mapsto r\bigl( x, \, \bH^{-1}(b) \bigr)$ is $M$--Lipschitz for each $b \in \mathcal{B}$.
All in all, given the boundedness of the $\phi_b$ and of $r$, the mapping $(x,y) \mapsto \om\bigl(x,\bH(y)\bigr)$ is also Lipschitz continuous.
Since taking the upper-right corner is a Lipschitz mapping as well (for the Hausdorff distance),
we get, by composition, the desired Lipschitz property for $\tr_H$.
\end{proof}

We are now ready to prove Proposition~\ref{PR:NonLinearApproach1}.
(Note that it needs a proof and that it is not implied by the various results
discussed in Section~\ref{sec:prodmeasureorthants}. Indeed, $(\tr_H,H)$ is an auxiliary game which, by construction, has the upper-right-corner
property, but $\tr_H$ is not linear, while linearity of the payoff function was a crucial feature of the setting
studied therein.)

\begin{proof}[Proof of Proposition~\ref{PR:NonLinearApproach1}]
We start with the direct implication and consider some containing half-space $\Chs$. The latter is parameterized by
$\alpha \in \R^d$ and $\beta \in \R$, and equals
\[
\Chs = \bigl\{ \omega  \in \R^d : \ \langle \alpha, \, \omega \rangle \leq \beta \bigr\}\,.
\]
Since $\Chs$ contains the orthant $\Corth(a)$, there are sequences  $(\omega_n)$ in $\Chs$ with components tending to
$-\infty$. Therefore, we necessarily have that $\alpha \succcurlyeq 0$. The convexity/concavity of $\tr_H$ in the sense of
$\preccurlyeq$ thus entails that the function $G_{\alpha,\beta} : (x,y) \longmapsto \bigl\langle \alpha,\,\tr_H(x,y) \bigr\rangle - \beta$
is also convex/concave. The continuity
of $G_{\alpha,\beta}$ follows from the one of $\tr_H$. The Sion--Fan lemma
applies and guarantees that
\[
\min_{x \in \Delta(\cI)} \, \max_{y \in \Delta(\cJ)} \, G_{\alpha,\beta}(x,y) =
\max_{y \in \Delta(\cJ)} \, \min_{x \in \Delta(\cI)} \, G_{\alpha,\beta}(x,y)\,,
\]
(the suprema and infima are all attained and are denoted by maxima and minima). Now, by assumption,
for all $y \in \Delta(\cJ)$, there exists $x \in \Delta(\cI)$ such that $\tr_H(x,y) \in \Corth(a)$. This means that
the above $\max \min G_{\alpha,\beta}$ is non-positive. Putting all things together, we have proved that
\[
\min_{x \in \Delta(\cI)} \, \max_{y \in \Delta(\cJ)} \, G_{\alpha,\beta}(x,y) \leq 0\,.
\]
That is, there exists $x_0 \in \Delta(\cI)$, e.g., the element attaining the above maximum,
such that for all $y \in \Delta(\cJ)$, one has $G_{\alpha,\beta}(x_0,y) \leq 0$, or, equivalently,
$\tr_H(x_0,y) \in \Corth(a)$. This property is exactly the stated one-shot $\tr_H$--approachability
of $\Corth(a)$. \medskip

Conversely, assume that there exists some $y_0 \in \Delta(\cJ)$ such that, for all $x \in \Delta(\cI)$, one has
$\tr_H(x,y_0) \not\in \Corth(a)$. By continuity of $\tr_H$ and closedness of $\Corth(a)$,
there exists some $\delta>0$ such that $d_{\Corth(a)}\bigl( \tr_H(x,y_0) \bigr) \geq \delta$ for all $x \in \Delta(\cI)$.
Now, as indicated around~\eqref{eq:distCorth}, the distance to $\Corth(a)$ is non-decreasing for $\preccurlyeq$.
In view of the convexity of $\tr_H$ in its first argument, this shows that we also have
$d_{\Corth(a)}(z) \geq \delta$ for all elements $z$ of the convex hull $\cC_{\tr_H,y_0}$
of the set $\bigl\{ \tr_H(x,y_0) : \,\, x \in \Delta(\cI) \bigr\}$. That is, the closed convex sets
$\cC_{\tr_H,y_0}$, which is compact, and $\Corth(a)$, which is closed, are disjoint and thus, by the Hahn--Banach theorem, are strictly separated
by some hyperplane. One of the two half-spaces thus defined, namely, the one not containing $\cC_{\tr_H,y_0}$,
is not one-shot $\tr_H$--approachable.
\end{proof}

\subsection*{A new approachability strategy of an orthant under partial monitoring}

Theorem~\ref{TH:orthant} suggests an approachability strategy based on surrogate payoffs and not only on the information gained, i.e.,
based on the mapping $\tr_H$ and not only on the signaling structure $H$ (and the estimated flags).
The first approach was already considered by~\cite{Koh75}
while other works, like \cite{Per11} and \cite{Ext}, resorted to the second one.
The considered strategy is an adaptation of Blackwell's strategy (which was recalled after the statement of
Theorem~\ref{CharacBset}): such an adaptation is possible as the latter strategy only relies
on the one-shot approachability of half-spaces, which is satisfied here with the surrogate payoffs $\tr_H$.

\paragraph{Description and convergence analysis of the strategy.}
As in the proof of Proposition~\ref{PR:Kohlberg}, we assume initially that flags $\bh_t = \bH(y_t)$ are observed
at the end of each round $t$ and that mixed payoffs are to be controlled.
The player then knows his mixed payoffs $\tr_{H,t} := \tr_H(x_t,y_t) = R\bigl(x_t,\bh_t\bigr)$
and aims at controlling his average payoffs, which we recall are denoted by $\overline{\tr}_{H,n}$.
Similarly to what was done in the proof of Proposition~\ref{PR:Kohlberg}, the one-shot $\tr_H$--approachability
of the containing half-spaces of $\Corth(a)$ entails that for each round $n$, there exists $x_{n+1} \in \Delta(\cI)$
such that
\[
\forall\, y \in \Delta(\cJ), \qquad
\left\langle \tr_H(x_{n+1},y)-\pi_{\Corth(a)}\bigl(\overline{\tr}_{H,n}\bigr),
\,\,\overline{\tr}_{H,n}-\pi_{\Corth(a)}\bigl(\overline{\tr}_{H,n}\bigr)\right\rangle
\leq 0\,.
\]
The sequence $\bigl(\overline{\tr}_{H,n}\bigr)$ thus satisfies Blackwell's condition
and as a result we get
\[
d_{\Corth(a)}\!\left( \frac{1}{n} \sum_{t=1}^n \tr_H(x_t,y_t) \right) \leq \frac{2M \sqrt{d}}{\sqrt{n}}\,.
\]
(See again the proof of Proposition~\ref{PR:Kohlberg} for this derivation and keep in mind that in the present setting
where the upper-right-corner property is not satisfied, $R$ is only bounded in $\ell_2$--norm by $M\sqrt{d}$.)
Since $r(x_t,y_t) \in \om\bigl(x_t,\bH(y_t)\bigr) \subseteq \Corth\bigl(\tr_H(x_t,y_t)\bigr)$,
we get $r(x_t,y_t) \preccurlyeq \tr_H(x_t,y_t)$ and, in view again of~\eqref{eq:distCorth},
\[
d_{\Corth(a)}\!\left( \frac{1}{n} \sum_{t=1}^n r(x_t,y_t) \right) \leq
d_{\Corth(a)}\!\left( \frac{1}{n} \sum_{t=1}^n \tr_H(x_t,y_t) \right)
\leq \frac{2M \sqrt{d}}{\sqrt{n}}\,.
\]

The same trick of playing i.i.d.\ in blocks as in the second part of
the proof of Proposition~\ref{PR:Kohlberg}, together with martingale convergence arguments,
relaxes the assumptions of flags being observed and payoffs being evaluated with mixed actions,
leading to the desired $(r,H)$--approachability strategy. (This is where we need
the Lipschitzness properties stated in Lemma~\ref{LM:tr} and its proof.) A more careful
study, which we omit here for the sake of brevity, shows that $(r,H)$--approachability
takes place at a $n^{-1/5}$--rate.

\paragraph{What we proved in passing.}
We proved in a constructive way that when an orthant is $(r,H)$--approachable, it is also $\bigl(\tr_H,\Full)$--approachable.

Conversely, assume that the equivalent conditions in Theorem~\ref{TH:orthant} are not satisfied, i.e., that
the orthant at hand, $\Corth(a)$, is not $(r,H)$--approachable. Then, (the proof of) Proposition~\ref{PR:NonLinearApproach1}
indicates that there exists some $y_0 \in \Delta(\cJ)$ such that the set $\Corth(a)$
and the convex hull of $\bigl\{ \tr_H(x,y_0), \,\, x \in \Delta(\cI) \bigr\}$ are strictly
separated. This implies in particular that $\Corth(a)$ is $\bigl(\tr_H,\Full\bigr)$--excludable, and thus is not
$\bigl(\tr_H,\Full\bigr)$--approachable.

Putting all things together, we have proved the following equivalence:
\begin{align*}
& \Corth(a) \ \mbox{is} \ (r,H)\mbox{--approachable} \\
\Longleftrightarrow \quad & \Corth(a) \ \mbox{is} \ \bigl(\tr_H,\Full\bigr)\mbox{--approachable} \\
\Longleftrightarrow \quad & \Corth(a) \ \mbox{is not} \ \bigl(\tr_H,\Full\bigr)\mbox{--excludable.}
\end{align*}
Note that the $\bigl(\tr_H,\Full\bigr)$--approachability is a form of non-linear approachability,
by which we mean that the function $\tr_H$ is not linear and yet, approachability is possible.
This result could be generalized (but we omit the description of the extension for the sake of concision).

\paragraph{On the computational complexity of the above-described strategy.}
The strategy we have exhibited reduces to solving, at each stage, a program
of the form
\[
\min_{x_{n+1} \in \Delta(\cI)} \biggl\{ \max_{y \in \Delta(\cJ)} \,\,
\bigl\langle \tr_H(x_{n+1},y) - \beta, \, \alpha \bigr\rangle \biggr\}
\]
for some vectors $\alpha,\beta \in \R^d$.
At first sight, it cannot be written as a finite linear program as $\tr_H$ is not a linear function of its arguments.
However, as proved in~\citet[Section~7.1]{Ext}, the function $\tr_H$ is actually
piecewise linear; that is, there exist some finite liftings of $\Delta(\cI)$ and $\Delta(\cJ)$ with respect to which $\tr_H$ is linear.
(These liftings only need to be computed once, before the game starts.)
Moreover, the per-step computational complexity of our strategy is constant (in fact, it is polynomial in the sizes of these
liftings; see \citealt{Ext} for more details).

\section{Primal approachability of polytopes}
\label{SE:Polyhedra}

Recall that a convex set $\Cpoly$ is a polytope if it is the intersection of a finite number of half-spaces $\bigl\{ \omega  \in \R^d : \ \langle \omega, a_\ell \rangle \leq b_\ell \bigr\}$, for $a_\ell,b_\ell \in \R^d$ and $\ell$ ranging in some finite set $\cL$. That is,
\begin{equation}
\label{eq:Cpoly}
\Cpoly = \bigcap_{\ell \in \cL} \left\{ \omega  \in \R^d : \ \langle \omega, a_\ell \rangle \leq b_\ell \right\} = \left\{ \omega  \in \R^d:
\ \max_{\ell \in \cL} \langle \omega, a_\ell \rangle - b_\ell \leq 0 \right\}.
\end{equation}
The following lemma (which is a mere exercice of rewriting)
states that an approachability problem of a polytope can be transformed into an approachability problem of some negative orthant.
We denote by $(0)_{\cL} = (0,\ldots,0)$ the null vector of $\R^{\cL}$. The negative orthant of $\R^{\cL}$ is then denoted by
$\Corth\bigl( (0)_{\cL} \bigr)$.
\begin{lemma}
\label{LM:equivalent}
The convex polytope $\Cpoly$ defined in~\eqref{eq:Cpoly}
is $(r,H)$--approachable if and only if the negative orthant $\Corth\bigl( (0)_{\cL} \bigr)$ is $(s,H)$--approachable, where
the vector-valued payoff function $s : \Delta(\cI) \times \Delta(\cJ) \to \R^{\cL}$ is defined as
\begin{equation}
\label{eq:s}
s(x,y) = \Big[\langle r(x,y), a_\ell \rangle - b_\ell  \Big]_{\ell \in \cL}\,.
\end{equation}
\end{lemma}
\begin{proof}
The result follows from the equivalence (see, e.g., property~3 in Appendix A.1 of~\citealt{Per11b}) of the distances
to $\Cpoly$ given by
\[
d_{\Cpoly} \qquad \mbox{and} \qquad d_{\Corth( (0)_{\cL} )}\bigl( T(\,\cdot\,)\bigr)\,,
\]
where $T : \R^d \to \R^{\cL}$ is the linear transformation
$\omega \mapsto \bigl[\langle \omega, a_\ell \rangle - b_\ell\bigr]_{\ell \in \cL}$.
\end{proof}

Theorem~\ref{TH:orthant} can then be rewritten, using Lemma~\ref{LM:equivalent} above,
to provide the desired primal characterization of polytopes.

\begin{corollary}
\label{cor:polyh}
Consider the convex polytope $\Cpoly$ given by~\eqref{eq:Cpoly},
together with the payoff function $s$ defined in~\eqref{eq:s}. Then,
\begin{align}
\label{EQ:proppolyhe}
& \Cpoly \ \mbox{is} \ (r,H)\mbox{--approachable} \\
\nonumber
\Longleftrightarrow \quad & \mbox{every containing half-space of} \ \Corth\bigl( (0)_{\cL} \bigr)
\ \mbox{is one-shot} \ \ts_H\mbox{--approachable.}
\end{align}
\end{corollary}

When $\Cpoly$ is indeed $(r,H)$--approachable, the results of the previous section provide
an approachability strategy of $\Corth\bigl( (0)_{\cL} \bigr)$ based on the transformed
payoffs $\ts_H$. This strategy also approaches $\Cpoly$ in view of Lemma~\ref{LM:equivalent},
however it might not be representable in the original space $\R^d$, as demonstrated in the following (counter-)example.

\begin{example}{}
\label{ex:hidden}
Consider on the one hand the polytope $\Cpoly = \bigl\{ \omega \in \R : \ \omega \in [-1,1] \bigr\}$
and the associated linear transformation $T$ defined by $T(\omega) = (\omega -1, \, -\omega-1) \in \R^2$ for all $\omega \in \R$.
Consider on the other hand the following game.
The sets of pure actions are $\cI=\{T,B\}$ and $\cJ=\{L,R\}$,
the signaling structure is $H = \Dark$ (with single signal denoted by $\emptyset$),
and the payoff function $r$ is given by the matrix
\[
\begin{array}{cccc}
\hline
& & L & R \\
\hline
T & & -1 & 2 \\
B & & -2 & 1 \\
\hline
\end{array}
\]
We identify $\Delta(\cI)$ and $\Delta(\cJ)$ with $[0,1]$.

We first discuss the dual condition~\eqref{EQ:charac} for $(r,\Dark)$--approachability.
For all $x \in [0,1]$, we have
$\om(x,\emptyset) = [-2+x,\, 1+x]$.
Thus, no mixed action $x$ of the player is such that
$\om(x,\emptyset)$ is included in $\Cpoly$,
which is therefore not $(r,\Dark)$--approachable.

We now turn to the primal condition as stated by Corollary~\ref{cor:polyh}.
We denote by $T = (T_1,T_2)$ the components of the linear transformation $T$.
From the linearity of $T$, we deduce from the above-stated expression of $\om$ (based on $r$) that
the sets of compatible payoffs in terms of $s = T(r)$ are of the form $T\bigl(\om(x,\emptyset)\bigr)$.
Taking the maxima, we thus get, for all mixed actions $x \in [0,1]$ (and all $y \in [0,1]$
as the game is played in the dark),
\[
\ts_{\Dark}(x,y) = \Bigl( \max T_1\bigl([-2+x,\, 1+x]), \,\, \max T_2\bigl([-2+x,\, 1+x]) \Bigr) =
(x,1-x)\,.
\]
Again, the necessary and sufficient condition for $(r,H)$--approachability of $\Cpoly$ fails,
as no containing half-space of the negative orthant but two of them is one-shot $\ts_{\Dark}$--approachable.
More precisely, these half-spaces are parameterized by $(p,1-p)$ where $p \in [0,1]$ and
correspond to the points $\bigl\{ (t_1,t_2) \in \R^2 : \ p_1 t_1 + p_2 t_2 \leq 0 \bigr\}$.
Except for the case when $p = 0$ or $p = 1$, these
half-spaces are strictly separated from the convex set $\bigl\{(x,1-x) : \ x \in [0,1]\bigr\}$.

The question now is whether we could have determined this by satisfying some primal condition
in the original space $\R$. First, consider some containing half-space of $\Cpoly$, typically,
either $(-\infty, 1]$ or $[-1,+\infty)$. Their transformations by $T$ into subsets of $\R^2$
are included respectively in $(-\infty, 0]\times \R$ or $\R\times(-\infty, 0]$.
These are precisely the only two half-spaces that were one-shot $\ts_{\Dark}$--approachable (by resorting
to one of the pure actions).
Now, and more importantly, consider the containing half-space of the negative orthant in $\R^2$ parameterized by
$p = 1/2$, that is, $\Chsind{1/2} = \bigl\{ (t_1,t_2) \in \R^2 : \ t_1 + t_2 \leq 0 \bigr\}$.
As indicated above, it is not one-shot $\ts_{\Dark}$--approachable. However, this
half-space contains all the original space, in the sense that $T(\R) \subset \Chsind{1/2}$,
as follows from  simple computations: $T(\omega) = (\omega-1) + (-\omega-1) = -2$. Therefore, there is no hope to prove,
based even on general subsets of the original game with payoffs in $\R$, that
the necessary and sufficient condition on the half-space $\Chsind{1/2}$ in the transformed space $\R^2$
fails.
\end{example}

The fundamental reason why the primal characterization in the transformed space cannot be
checked based on considerations in the original space is the following. In the absence of a upper-right-corner property,
the range of $\ts_{\Dark}$ is outside the range of $s$ but we can only access to the latter based on
the original space.
The moral of this example is that we have to consider some \emph{hidden}
containing half-spaces of the polytope $\Cpoly$ in order to establish some primal characterization: this is precisely what
Condition~(\ref{EQ:proppolyhe}) does.

\section{Primal approachability of general convex sets}
\label{sec:gen}

We consider in this section the primal approachability of general closed convex sets $\cC$.
In the case of polytopes, Lemma~\ref{LM:equivalent} was essentially indicating that only
finitely many directions in $\R^d$ (the ones given by the $a_{\ell}$) need be considered.
In the case of general convex sets, all directions are to be studied. We do so by
resorting to support functions, which we define based on the unit Euclidean sphere
$\cS = \bigl\{ \omega \in \R^d : \ \|\omega\| = 1 \bigr\}$.
More formally, the support function $\phi_{\cC} : \cS \to \R \cup\{+\infty\}$ of a set $\cC \subseteq \R^d$ is
defined by
\[
\forall s \in \cS, \qquad \phi_{\cC}(s) = \sup \bigl\{ \langle c, s \rangle : \ c \in \cC \bigr\}\,.
\]

We now construct a lifted setting in which one-shot approaching the containing half-spaces for some
payoff function will be equivalent to $(r,H)$--approaching the original closed convex set $\cC$.
This setting is given by some set of integrable functions on $\cS$. We equip the latter with
the (induced) Lebesgue measure, for which $\cS$ has a finite measure.
That is, we consider the set $\L^2(\cS)$ of Lebesgue square integrable functions $\cS \to \R$,
equipped with the inner product
\[
(f,g) \in \L^2(\cS) \times \L^2(\cS) \, \longmapsto \, \int_{\cS} f(s)\,g(s)\,\d s\,.
\]
The orthant in $\L^2(\cS)$ corresponding to $\cC \subseteq \R^d$ is
\[
\Corth(\phi_{\cC}) = \bigl\{ f \in \L^2(\cS) : \ f \leq \phi_\cC \bigr\}\,.
\]
The description of the lifted setting is concluded by stating the considered payoff function
$\Phi : \Delta(\cI) \times \Delta(\cJ) \to \L^2(\cS)$. It indicates, as in the previous
sections, how to transform payoffs given the signaling structure $H$. Formally,
\[
\forall\, x \in \Delta(\cI), \ \forall\, y \in \Delta(\cJ), \qquad  \Phi(x,y) = \phi_{\om(x,\bH(y))}\,.
\]
The square integrability of $\Phi(x,y)$ follows its boundedness, which itself stems from the boundedness of $\om\bigl(x,\bH(y)\bigr)$.
(See Lemma~\ref{lm:support} in appendix, property~\ref{item:bd}, for a reminder of this well-known result and others on support functions.)

We are now ready to state the primal characterization of approachability with partial monitoring in the general case.

\begin{theorem}
\label{TH:primalgeneral}
For all games $(r,H)$, for all closed convex sets $\cC \subset \R^d$,
\begin{align*}
& \cC \ \mbox{is} \ (r,H)\mbox{--approachable} \\
\Longleftrightarrow \quad  & \mbox{every half-space $\Chs \supset \Corth(\phi_{\cC})$
is one-shot} \ \Phi\mbox{--approachable}.
\end{align*}
\end{theorem}

\begin{proof}
We first note that we can assume with no loss of generality that $\cC$ is bounded thus compact.
Indeed, $\cC$ is $(r,H)$--approachable if and only if its
intersection $\cC \cap r\bigl(\Delta(\cI \times \cJ)\bigr)$
with the bounded convex set of feasible payoffs is approachable. This entails that $\phi_{\cC} \in \L^2(\cS)$.

Now, the proof follows along the lines of the proof of Theorem~\ref{TH:orthant}. In particular,
we exploit the dual characterization~\eqref{EQ:charac}, that indicates that for all $y \in \Delta(\cJ)$,
there exists $x \in \Delta(\cI)$ such that $\om\bigl(x,\bH(y)\bigr) \subseteq \cC$.
It can be restated equivalently (see Lemma~\ref{lm:support} in appendix, property~\ref{item:incl}) as stating that
for all $y \in \Delta(\cJ)$, there exists $x \in \Delta(\cI)$ such that $\Phi(x,y) \leq \phi_{\cC}$.
We thus only need to show that the stated primal characterization is equivalent to the latter condition.
\medskip

We start with the direct implication (from the dual condition to the primal condition).
As recalled in the proof of Lemma~\ref{LM:tr}, the function $\om$ is concave/convex, which, together
with properties~\ref{item:incl} and~\ref{item:lin} of Lemma~\ref{lm:support}, shows that
$\Phi$ is also convex/concave.
Moreover, as proved at the end of the proof of Lemma~\ref{LM:tr},
the function $(x,y) \mapsto \om\bigl(x,\bH(y)\bigr)$ is a Lipschitz function, with Lipschitz constant denoted by
$L_{\om}$, when the set of compact subsets of the Euclidean ball of $\R^d$ with center $(0,\ldots,0)$
and radius $M$ is equipped with the Hausdorff distance. This entails that $\Phi$ is also a Lipschitz function,
with constant $L_{\om} V$, where $V$ is the volume of $\cS$ for the induced Lebesgue measure.
This is because the Hausdorff distance $\delta$ between two sets $\cD_1$ and $\cD_2$
translates to a $V\delta$--Euclidean distance between $\phi_{\cD_1}$ and $\phi_{\cD_2}$. Indeed, we have,
by definition of the Hausdorff distance, $\cD_1 \subseteq \cD_2 + B_{\delta}$ and
$\cD_2 \subseteq \cD_1 + B_{\delta}$, where $B_{\delta}$ is the Euclidian ball of $\R^d$
with center $(0,\ldots,0)$ and radius~$\delta$. Properties~\ref{item:lin} and~\ref{item:bd} of Lemma~\ref{lm:support}
respectively yield the inequalities
\[
\bigl| \phi_{\cD_1} - \phi_{\cD_2} \bigr| =
\max \bigl\{ \phi_{\cD_1} - \phi_{\cD_2}, \,\, \phi_{\cD_2} - \phi_{\cD_1} \bigr\} \leq \phi_{B_\delta}
\leq \delta\,,
\]
with, by integration, $\bigl\| \phi_{\cD_1} - \phi_{\cD_2} \bigr\| \leq V\delta$.

The above-stated properties of $\Phi$ imply that for all $\psi \in \L^2(\cS)$ with $\psi \geq 0$, the game
$(x,y) \mapsto \langle \psi, \, \Phi(x,y) \rangle$ has a value $v(\psi)$,
and that this value is achieved: there exists some $x_\psi \in \Delta(\cI)$ such that
\[
\max_{y \in \Delta(\cJ)} \bigl\langle \psi, \, \Phi(x_\psi,y) \bigr\rangle = v(\psi) =
\max_{y \in \Delta(\cJ)} \min_{x \in \Delta(\cI)} \bigl\langle \psi, \, \Phi(x,y) \bigr\rangle\,.
\]
Now, consider
some half-space $\Chs$ containing $\Corth\bigl( \phi_{\cC} \bigr)$. It is of the form
\[
\Chs = \bigl\{ f \in \L^2(\cS) : \ \langle \psi, f \rangle \leq \beta \bigr\}\,,
\]
where necessarily, as can be shown by contradiction,
$\psi \geq 0$. The dual condition is satisfied by assumption, that is,
for all $y \in \Delta(\cJ)$, there exists $x \in \Delta(\cI)$ such that $\Phi(x,y) \in \Corth(\phi_{\cC})$,
and therefore, $\Phi(x,y) \in \Chs$. Thus, $v(\psi) \leq \beta$, as can be seen with its expression
as a max/min. The mixed action $x_{\psi}$ thus satisfies that $\bigl\langle \psi, \, \Phi(x_\psi,y) \bigr\rangle
\leq \beta$ for all $y \in \Delta(\cJ)$, which is exactly saying that $\Phi(x_\psi,y) \in \Chs$
for all $y \in \Delta(\cJ)$. We therefore proved the desired one-shot $\Phi$--approachability of $\Chs$.
\medskip

Conversely, we assume that the dual condition is not satisfied, i.e.,
that there exists some $y_0 \in \Delta(\cJ)$ such that for no $x \in \Delta(\cI)$
one has $\Phi(x,y_0) \in \Corth(\phi_{\cC})$. We consider the continuous thus compact
image $\Phi\bigl(\Delta(\cI),y_0\bigr)$ of $\Delta(\cI)$ by $\Phi(\,\cdot\,,y_0)$.
Its Euclidean distance to the closed set $\Corth(\phi_{\cC})$ is thus positive, we denote it
by $\delta > 0$. Now, the distance of an element $f \in \L^2(\cS)$ to $\Corth(\phi_{\cC})$
is given by
\[
d_{\Corth(\phi_{\cC})}(f) = \int_{\cS} \bigl( f(s) - \phi_{\cC}(s) \bigr)_+ \,\d s\,.
\]
Since in addition, $\Phi$ is convex in its first argument (as shown in the first part of this
proof), we have that $d_{\Corth(\phi_{\cC})}(f) \geq \delta$ not only for all $f \in \Phi\bigl(\Delta(\cI),y_0\bigr)$
but also for all $f$ in the convex hull of $\Phi\bigl(\Delta(\cI),y_0\bigr)$.
The latter set is pointwise bounded (by $M$, as follows from property~\ref{item:bd} of Lemma~\ref{lm:support})
and is formed by equicontinuous functions (they all are $M$--Lipschitz continuous,
as follows from property~\ref{item:Lip} of the same lemma). The Arzela--Ascoli theorem
thus ensures that the closure of this set is compact for the supremum norm $\|\,\cdot\,\|_\infty$ over $\cS$.
As by integration $\|\,\cdot\,\|_\infty \geq \|\,\cdot\,\|/V$, the closure of the convex hull of
$\Phi\bigl(\Delta(\cI),y_0\bigr)$ and the set $\Corth(\phi_{\cC})$ are still $\delta/V$--separated in
$\|\,\cdot\,\|_\infty$--norm, thus are disjoint. Since the former set is a convex and compact set,
and the latter is a closed convex set, the Hahn--Banach theorem entails that they are strictly separated
by some hyperplane. In particular, one of the two thus-defined half-spaces is not one-shot $\Phi$--approachable.
\end{proof}

\subsection*{The above result is a generalization of the polytopial case}

In Section~\ref{SE:Polyhedra} we showed that when approaching a polytope, there are only finitely many directions (i.e.,
finitely many elements of the sphere $\cS$) of interest, namely, the directions corresponding to the defining hyperplanes.
The results we obtained therein can in fact be obtained as a corollary of Theorem~\ref{TH:primalgeneral} when the latter
is stated (and proved) with a different measure instead of the Lebesgue measure, given by
the sums of the Dirac masses on the directions of the defining hyperplanes.

There are two ways to extend the primal characterization of approachability under partial monitoring
from polytopes to general convex sets. The one we worked out above relies on the observation that
with general convex sets, every direction might be relevant, as a general convex set is defined as
the intersection of infinitely many half-spaces, one per element of $\cS$.
Based on this, we introduced for general convex sets a infinite-dimensional lifting into the space of real-valued mappings on the whole set $\cS$.
We also resorted the uniform Lebesgue measure since all directions are equally important.

The other way of generalizing the results relies on the fact that a closed convex set $\cC \subseteq \R^d$ is approachable
if and only if all containing polytopes are approachable. By playing in blocks and
approximating a given general convex set by a sequence of containing polytopes,
one could have shown that $\cC$ is $(r,H)$--approachable if and only
if all containing polytopes satisfy the characterization of Corollary~\ref{cor:polyh}.
However, while this alternative way leads to a characterization, it is less intrinsic
as there is no fixed lifted space to be considered. (The finite-dimensional lifted spaces depend
on the approximating polytopes.) For the sake of elegance, we thus used the
infinite-dimensional lifting described above.

\subsection*{Acknowledgements}
Shie Mannor was partially supported by the Israel Science Foundation under grant no.~920/12.
Vianney Perchet acknowledges support by ``Agence Nationale de la Recherche,'' under grant JEUDY (ANR-10-BLAN 0112).

\bibliographystyle{plainnat}
\bibliography{Bib-Kohl-ind}

\section*{Appendix}

\subsection*{A brief survey of some well-known properties of support functions}

For the sake of self-completeness only we summarize in the lemma below some
simple and well-known properties of support functions.

\begin{lemma} We consider a set $\cC \subseteq \R^d$.
\label{lm:support}
\begin{enumerate}
\item \label{item:bd} If $\cC$ is bounded in Euclidian norm by $C$, then $\phi_{\cC}$ is
bounded in supremum norm by $C$ and in Euclidean norm by $VC$,
where $V$ is the volume of $\cS$ under the (induced) Lebesgue measure.
\item \label{item:Lip} If $\cC$ is bounded in Euclidian norm by $C$, then $\phi_{\cC}$ is
a Lipschitz function, with Lipschitz constant $C$.
\item \label{item:incl} For all $\cC' \subseteq \R^d$, if $\cC \subseteq \cC'$, then $\phi_{\cC} \leq \phi_{\cC'}$.
The converse implication holds if in addition $\cC'$ is a closed convex set.
\item \label{item:lin} The function $\phi$ is linear, in the sense that for all $\gamma \geq 0$ and all
all $\cC' \subseteq \R^d$, one has $\phi_{\gamma\cC + \cC'} = \gamma \phi_{\cC} + \phi_{\cC'}$.
\end{enumerate}
\end{lemma}

\begin{proof}
Property~\ref{item:bd} follows from the Cauchy--Schwarz inequality:
for all $s \in \cS$,
\[
\bigl| \phi_{\cC}(s) \bigr| \leq \sup_{c \in \cC} \, \bigl| \langle c, s \rangle \bigr|
\leq \sup_{c \in \cC} \| c \| \, \| s \| = \sup_{c \in \cC} \| c \|\,,
\]
as the elements $s \in \cS$ have unit norm. The bound in Euclidean norm
follows by integration over $\cS$.

For property~\ref{item:Lip}, we note that $s \in \cS \mapsto \langle c, s \rangle$
is a $\|c\|$--Lipschitz function (again, by the Cauchy--Schwarz inequality). Therefore
$\phi_{\cC}$ is the supremum of $C$--Lipschitz functions and as such is also
a $C$--Lipschitz function.

The first part of property~\ref{item:incl} holds by the definition of a supremum.
To prove the converse implication, we use an argument by contradiction. We
consider two sets $\cC$ and $\cC'$, where $\cC'$ is closed and convex.
We assume that $\cC$ is not included in $\cC'$
and show that the existence of a $s \in \cS$ such that $\phi_{\cC}(s) > \phi_{\cC'}(s)$.
The set $\cC \setminus \cC'$ is not empty, let $x$ be one of its elements.
The convex sets $\{x\}$, which is compact, and $\cC'$, which is closed, are disjoint
sets. The Hahn--Banach theorem ensures the existence of a strictly separating hyperplane
between these convex sets,
which we can write in the form $\bigl\{ \omega \in \R^d : \ \langle \omega,s \rangle = \beta \bigr\}$
for some $s \in \cS$ and $\beta \in \R$ such that
\[
\phi_{ \{x\} }(s) = \langle x,s \rangle > \beta
\qquad \mbox{and} \qquad
\forall c' \in \cC', \quad \langle c',s \rangle < \beta\,.
\]
This entails that
$\phi_{\cC'}(s) \leq \beta < \phi_{ \{x\} }(s) \leq \phi_{\cC}(s)$.

Finally, the last property is true because by definition
\[
\gamma\cC + \cC' = \bigl\{ \gamma c + c' : \ c \in \cC, \ c' \in \cC' \bigr\}
\]
and thus, for all $s \in \cS$,
\[
\sup_{c'' \in \gamma\cC + \cC'} \langle c'', s \rangle
= \sup_{c \in \cC, \, c' \in \cC'} \gamma \langle c, s \rangle + \langle c', s \rangle
= \gamma \, \sup_{c \in \cC} \langle c, s \rangle + \sup_{c' \in \cC'} \langle c', s \rangle\,,
\]
where we used the fact that $\gamma > 0$ in the last equality.
\end{proof}

\end{document}